\documentclass[12pt]{amsart}
\usepackage{amssymb}
\usepackage{mathpazo}
\usepackage{fourier}
\usepackage{paralist}
\usepackage{faktor}
\batchmode

\usepackage[utf8]{inputenc}
\usepackage{hyperref}

\begin{document}

\def\dbl{[\hskip -1pt[}
\def\dbr{]\hskip -1pt]}
\title[Unique jet determination and extension  of germs of CR maps]{Unique jet determination  and   extension \\of germs  of CR maps  into spheres}

\author{Nordine Mir}

\address{ Texas A\&M University at Qatar, Science program, PO Box 23874, Education City, Doha, Qatar
}

\email{nordine.mir@qatar.tamu.edu}

\author{Dmitri Zaitsev}
\address{School of Mathematics, Trinity College Dublin, Dublin 2, Ireland}
\email{zaitsev@maths.tcd.ie}


\subjclass[2010]{32D15, 32V20, 32H02, 32H04, 32H12, 32V25, 32V40}
\keywords{CR maps, finite jet determination, holomorphic extension}



\def\Label#1{\label{#1}}

\def\1#1{\ov{#1}}
\def\2#1{\widetilde{#1}}
\def\6#1{\mathcal{#1}}
\def\4#1{\mathbb{#1}}
\def\3#1{\widehat{#1}}
\def\7#1{\mathscr{#1}}
\def\K{{\4K}}
\def\LL{{\4L}}

\def \MM{{\4M}}

\def \B{{\4B}^{2N'-1}}

\def \H{{\4H}^{2l-1}}

\def \F{{\4H}^{2N'-1}}

\def \LL{{\4L}}

\def\Re{{\sf Re}\,}
\def\Im{{\sf Im}\,}
\def\id{{\sf id}\,}

\def\s{s}
\def\k{\kappa}
\def\ov{\overline}
\def\span{\text{\rm span}}
\def\ad{\text{\rm ad }}
\def\tr{\text{\rm tr}}
\def\xo {{x_0}}
\def\Rk{\text{\rm Rk\,}}
\def\sg{\sigma}
\def \emxy{E_{(M,M')}(X,Y)}
\def \semxy{\scrE_{(M,M')}(X,Y)}
\def \jkxy {J^k(X,Y)}
\def \gkxy {G^k(X,Y)}
\def \exy {E(X,Y)}
\def \sexy{\scrE(X,Y)}
\def \hn {holomorphically nondegenerate}
\def\hyp{hypersurface}
\def\prt#1{{\partial \over\partial #1}}
\def\det{{\text{\rm det}}}
\def\wob{{w\over B(z)}}
\def\co{\chi_1}
\def\po{p_0}
\def\fb {\bar f}
\def\gb {\bar g}
\def\Fb {\ov F}
\def\Gb {\ov G}
\def\Hb {\ov H}
\def\zb {\bar z}
\def\wb {\bar w}
\def \qb {\bar Q}
\def \t {\tau}
\def\z{\chi}
\def\w{\tau}
\def\Z{{\mathbb Z}}
\def\phi{\varphi}
\def\eps{\epsilon}

\def \T {\theta}
\def \Th {\Theta}
\def \L {\Lambda}
\def\b {\beta}
\def\a {\alpha}
\def\o {\omega}
\def\l {\lambda}

\def \im{\text{\rm Im }}
\def \re{\text{\rm Re }}
\def \Char{\text{\rm Char }}
\def \supp{\text{\rm supp }}
\def \codim{\text{\rm codim }}
\def \Ht{\text{\rm ht }}
\def \Dt{\text{\rm dt }}
\def \hO{\widehat{\mathcal O}}
\def \cl{\text{\rm cl }}
\def \bS{\mathbb S}
\def \bK{\mathbb K}
\def \bD{\mathbb D}
\def \bC{\mathbb C}
\def \bL{\mathbb L}
\def \bZ{\mathbb Z}
\def \bN{\mathbb N}
\def \scrF{\mathcal F}
\def \scrK{\mathcal K}
\def \mc #1 {\mathcal {#1}}
\def \scrM{\mathcal M}
\def \cR{\mathcal R}
\def \scrJ{\mathcal J}
\def \scrA{\mathcal A}
\def \scrO{\mathcal O}
\def \scrV{\mathcal V}
\def \scrL{\mathcal L}
\def \scrE{\mathcal E}
\def \hol{\text{\rm hol}}
\def \aut{\text{\rm aut}}
\def \Aut{\text{\rm Aut}}
\def \J{\text{\rm Jac}}
\def\jet#1#2{J^{#1}_{#2}}
\def\gp#1{G^{#1}}
\def\gpo{\gp {2k_0}_0}
\def\emmp {\scrF(M,p;M',p')}
\def\rk{\text{\rm rk\,}}
\def\Orb{\text{\rm Orb\,}}
\def\Exp{\text{\rm Exp\,}}
\def\Span{\text{\rm span\,}}
\def\d{\partial}
\def\D{\3J}
\def\pr{{\rm pr}}

\def \CZZ {\C \dbl Z,\zeta \dbr}
\def \D{\text{\rm Der}\,}
\def \Rk{\text{\rm Rk}\,}
\def \CR{\text{\rm CR}}
\def \ima{\text{\rm im}\,}
\def \I {\mathcal I}

\def \M {\mathcal M}

\newtheorem{Thm}{Theorem}[section]
\newtheorem{Cor}[Thm]{Corollary}
\newtheorem{Pro}[Thm]{Proposition}
\newtheorem{Lem}[Thm]{Lemma}
\newtheorem{Def}[Thm]{Definition}
\newtheorem{Conj}[Thm]{Conjecture}
\newtheorem{Rem}[Thm]{Remark}

\theoremstyle{remark}

\newtheorem{Exa}[Thm]{Example}
\newtheorem{Exs}[Thm]{Examples}

\numberwithin{equation}{section}

\newcommand{\crb}{\mathcal{V}}
\newcommand{\dbar}{\bar\partial}
\newcommand{\genmat}{\lambda}
\newcommand{\polynorm}[1]{{|| #1 ||}}
\newcommand{\vnorm}[1]{\left\|  #1 \right\|}
\newcommand{\asspol}[1]{{\mathbf{#1}}}
\newcommand{\Cn}{\mathbb{C}^n}
\newcommand{\Cd}{\mathbb{C}^d}
\newcommand{\Cm}{\mathbb{C}^m}
\newcommand{\C}{\mathbb{C}}
\newcommand{\CN}{\mathbb{C}^N}
\newcommand{\CNp}{\mathbb{C}^{N^\prime}}
\newcommand{\Rd}{\mathbb{R}^d}
\newcommand{\Rn}{\mathbb{R}^n}
\newcommand{\RN}{\mathbb{R}^N}
\newcommand{\R}{\mathbb{R}}
\newcommand{\bR}{\mathbb{R}}
\newcommand{\N}{\mathbb{N}}
\newcommand{\dop}[1]{\frac{\partial}{\partial #1}}
\newcommand{\dopt}[2]{\frac{\partial #1}{\partial #2}}
\newcommand{\vardop}[3]{\frac{\partial^{|#3|} #1}{\partial {#2}^{#3}}}
\newcommand{\br}[1]{\langle#1 \rangle}
\newcommand{\infnorm}[1]{{\left\| #1 \right\|}_{\infty}}
\newcommand{\onenorm}[1]{{\left\| #1 \right\|}_{1}}
\newcommand{\deltanorm}[1]{{\left\| #1 \right\|}_{\Delta}}
\newcommand{\omeganorm}[1]{{\left\| #1 \right\|}_{\Omega}}
\newcommand{\nequiv}{{\equiv \!\!\!\!\!\!  / \,\,}}
\newcommand{\bk}{\mathbf{K}}
\newcommand{\p}{\prime}
\newcommand{\tV}{\mathcal{V}}
\newcommand{\poly}{\mathcal{P}}
\newcommand{\ring}{\mathcal{A}}
\newcommand{\ringk}{\ring_k}
\newcommand{\ringktwo}{\mathcal{B}_\mu}
\newcommand{\germs}{\mathcal{O}}
\newcommand{\On}{\germs_n}
\newcommand{\mcl}{\mathcal{C}}
\newcommand{\formals}{\mathcal{F}}
\newcommand{\Fn}{\formals_n}
\newcommand{\autM}{{\Aut (M,0)}}
\newcommand{\autMp}{{\Aut (M,p)}}
\newcommand{\holmaps}{\mathcal{H}}
\newcommand{\biholmaps}{\mathcal{B}}
\newcommand{\autmaps}{\mathcal{A}(\CN,0)}
\newcommand{\jetsp}[2]{ G_{#1}^{#2} }
\newcommand{\njetsp}[2]{J_{#1}^{#2} }
\newcommand{\jetm}[2]{ j_{#1}^{#2} }
\newcommand{\glnc}{\mathsf{GL_n}(\C)}
\newcommand{\glmc}{\mathsf{GL_m}(\C)}
\newcommand{\glc}{\mathsf{GL_{(m+1)n}}(\C)}
\newcommand{\glk}{\mathsf{GL_{k}}(\C)}
\newcommand{\smC}{\mathcal{C}^{\infty}}
\newcommand{\anC}{\mathcal{C}^{\omega}}
\newcommand{\kC}{\mathcal{C}^{k}}
\newcommand{\fps}[1]{\C\llbracket #1 \rrbracket}
\newcommand{\fpstwo}[2]{#1\llbracket #2 \rrbracket}
\newcommand{\cps}[1]{\C\{#1\}}
\newcommand{\cinfty}{\6C^\infty}
\newcommand{\diffable}[1]{\6C^{#1}}
\newcommand{\idealsheaf}{\6I}
\newcommand{\note}[1]{#1}
\newcommand{\inp}[1]{\langle #1 \rangle}
\newcommand*{\newfaktor}[2]{
  \raisebox{0.5\height}{\ensuremath{#1}}
  \mkern-5mu\diagup\mkern-4mu
  \raisebox{-0.5\height}{\ensuremath{#2}}
}



\begin{abstract} 
We provide a new way of simultaneously parametrizing  arbitrary local CR maps from real-analytic generic manifolds $M\subset\C^N$ into spheres $\4S^{2N'-1}\subset \C^{N'}$ of any dimension. The parametrization is obtained as a composition of universal rational maps with a holomorphic map depending only on $M$. As applications, we obtain rigidity results of different flavours
such as unique jet determination
and global extension of local CR maps.



 \end{abstract}

\maketitle

\section{Introduction}

In 1907, Poincar\'e \cite{Po} discovered the first remarkable geometric properties of local biholomorphic mappings sending real hypersurfaces in multidimensional complex space into each other. His work, together with the later work of  Cartan \cite{Cartan}, Tanaka \cite{Tan} and Chern-Moser \cite{CM} unveiled 
the striking strong rigidity properties that such maps possess. Among such properties, of particular interest to us in this paper are those of uniqueness and extension.  

It follows from \cite{Cartan, Tan, CM} that local real-analytic CR diffeomorphims, i.e. local biholomorphic mappings, between Levi-non degenerate real-analytic hypersurfaces in $\C^N$ are  uniquely determined by their 2-jets at any 
fixed point of the source hypersurface.  Subsequent work over the last decades has been  devoted to understand to what extent such a uniqueness property were true in further generality. In \cite{ELZ, BERcag, BER00, BMR, Juhlin}, optimal conditions for the finite jet determination property to hold for CR automorphisms between general real-analytic hypersurfaces, or CR submanifolds of higher codimension, have been found. For Levi-degenerate CR manifolds, a number of results have been obtained exploring the relationship between the jet order required to get uniqueness  and the geometry of the manifolds, see \cite{ELZ, LM07, JL13, KMZ} and the references therein. In another direction, the above mentioned uniqueness result due to Tanaka, Cartan and Chern-Moser for CR diffeomorphisms  has recently been shown to hold for sufficiently smooth CR manifolds, see \cite{Ebasian, KZ, BB, BBM, BDL}. 

The present paper proposes a universal parametrization tool (Proposition~\ref{p:leader} below) 
with applications including the finite jet determination and global extension problems for local holomorphic mappings sending real-analytic CR submanifolds embedded in complex spaces of different dimension. In contrast to the biholomorphic setting, these are largely unexplored territories besides the case of local CR maps  between spheres  $\C^N \supset \4S^{2N-1} \to \4S^{2N'-1}\subset \C^{N'}$, where $N,N'\geq 2$. Indeed, by the work of Forstneri\v c \cite{F89}, such maps extend automatically as (global) rational maps  with a uniform bound on their degree; as a consequence, unique determination by a finite jet (at any fixed point) necessarily holds for such maps. However, the above strategy becomes no longer available when tackling CR embeddings from real-analytic hypersurfaces into spheres, as  the mappings under consideration need not be rational. And despite of having been subject of many related studies (see e.g. \cite{EHZ1, EHZ2, EL}), it is still an open question to decide whether finite jet determination holds in the latter setting. In this paper, we will answer this by the affirmative by proving the following more general result:

\begin{Thm}\label{t:mainbest}
Let $M\subset \C^N$ be a real-analytic minimal CR submanifold.  Then for every point $p\in M$, there exists an integer $k=k(p)$ such that if $f,g\colon (M,p)\to \4S^{2N'-1}$ are two germs of $\6C^\infty$-smooth CR maps with $j^{k}_pf=j^{k}_pg$, then $f=g$. Furthermore, the map  $M\ni p\mapsto k(p)$ may be chosen to be  bounded on compact subsets of $M$.
\end{Thm}

Recall here that $M$ is called minimal (in the sense of \cite{T90}) if $M$ does not contain any proper CR submanifold of the same CR dimension as that of $M$. Since compact real-analytic real hypersurfaces are always minimal (see \cite{DF}), we immediately obtain the following:

\begin{Cor}\label{c:compact}
 For every compact real-analytic hypersurface $M\subset \C^N$, there exists an integer $\ell=\ell(M)$ such that if $f,g\colon (M,p)\to \4S^{2N'-1}$ are two germs of $\6C^\infty$-smooth CR maps at some point $p\in M$ with $j^{\ell}_pf=j^{\ell}_pg$, it follows that $f=g$.
\end{Cor}

Corollary \ref{c:compact} can be applied to get the following boundary uniqueness theorem for proper holomorphic mappings into balls of higher dimension.

\begin{Cor}\label{c:proper}
 Let $\Omega \subset \C^N$ be  a  bounded domain with smooth real-analytic boundary and $\4B^{N'}\subset \C^{N'}$ be the unit ball. Then there exists an  integer $\ell$, depending only on $\partial \Omega$, such that if $F,G\colon \Omega \to \4B^{N'}$ are two proper holomorphic mappings extending smoothly up to the boundary near some point $p\in \partial \Omega$ with $j_p^\ell F=j^\ell_pG$, it follows that $F=G$.
\end{Cor}

We will  establish Theorem \ref{t:mainbest} (as well as Corollaries \ref{c:compact} and \ref{c:proper})  for local holomorphisms, since all $\6C^\infty$-smooth CR maps under consideration automatically extend holomorphically to a neighborhood of $p$ in $\C^N$ according to \cite{MMZ03,Mi03}. 

As mentioned above, we apply the same universal parametrization tool (Proposition~\ref{p:leader}) 
to study the independent question about global extension of germs of CR maps.
We shall prove:

\begin{Thm}\label{t:meromorphic} Let $M\subset \C^N$ be a real-analytic generic minimal submanifold. Then for every point $p_0\in M$, there exists a neighborhood $\Omega$ of $p_0$ in $\C^N$ such that for every $q\in \Omega\cap M$, any germ $f\colon (M,q)\to \4S^{2N'-1}$ of a $\6C^\infty$-smooth CR map extends meromorphically to $\Omega$. Furthermore, if $M$ is a real  hypersurface, the meromorphic extension over $\Omega$ is in fact holomorphic.
\end{Thm}

Using standard analytic continuation arguments (see  \S \ref{s:final}), Theorem \ref{t:meromorphic} provides the following global extension result.

\begin{Cor}\label{c:global}
Let $M\subset \C^N$ be a real-analytic
hypersurface
that is both connected and simply-connected and
 contains no complex-analytic hypersurface of $\C^N$.
Then for every point $p_0\in M$, any germ of a $\6C^\infty$-smooth CR map $f\colon (M,p_0)\to \4S^{2N'-1}$ extends holomorphically to a neighborhood of $M$ in $\C^N$. If, moreover, $M$ does not contain any positive dimensional complex-analytic subvariety, the same conclusion holds for merely $\6C^{N'-N+1}$-smooth CR maps.
\end{Cor}

There is a substantial literature related to Corollary \ref{c:global}.  Poincar\'e \cite{Po} was the first  to prove that any local biholomorphic map sending a piece of the sphere in $\C^2$ into itself extends as a global holomorphic map between the corresponding unit balls. Poincar\'e's global extension phenomenon was  later extended by Tanaka \cite{Tan} and Alexander \cite{Alex} for spheres in arbitrary dimension. Further generalizations of this extension property for local biholomorphisms have been investigated: extension along paths of local maps between strongly pseudoconvex real-analytic hypersurfaces  was considered in \cite{Pin, VEK}; for algebraic real hypersurfaces, or even  CR manifolds of higher codimension,  general extension results as algebraic maps have been established in \cite{W77, HJ98, BERacta}.  In contrast, Theorem \ref{t:meromorphic} and Corollary \ref{c:global}  address the global extension problem for local holomorphisms of positive codimension, on which much less is known. In that regard, the reader should note that  Theorem \ref{t:meromorphic} is equivalent to an extension result along any path starting from $p_0$. Hence, in the case where $M$ is a strictly pseudoconvex real hypersurface, Theorem \ref{t:meromorphic} recovers a result from \cite{Pi}. Note that in the case where $M$ is also a sphere, Corollary \ref{c:global} follows from the rationality result given in  \cite{F89} and \cite{CS}. We should mention that for maps between spheres, hyperquadrics, or boundaries of bounded symmetric domains, global holomorphic extension of local holomorphic maps may follow from more general results known as "rigidity"  theorems. The reader is referred to the papers \cite{W79, Faran, H99, H03, BH05, KZinvent, KZma} and  the survey paper \cite{HYsurvey} for more on this specific topic.

The main novelty of the present work consists of providing a unified framework that allows us  to study, at the same time and despite of being very different in nature, the finite jet determination and global extension problems for local holomorphic maps. After collecting   some preliminary results and notation in \S \ref{s:notation}, we explain the details of our approach in  \S \ref{s:meromorphicparam}; we prove that germs of CR maps as in Theorem \ref{t:mainbest} can be universally meromorphically parametrized by their jets at a generic point. In fact, one needs a very precise statement indicating the (real-analytic) dependence on the base point where the germ is defined, see Proposition \ref{p:leader}. Such a result not only allows us to understand the structure of germs of CR maps whose base point changes, but is also crucial in order to choose a jet order $k(p)$ (as in Theorem \ref{t:mainbest}) that remains bounded on compact subsets of $M$. The proofs of Theorems \ref{t:mainbest} and \ref{t:meromorphic} and Corollary \ref{c:global} are finalized in \S \ref{s:final}.







\section{Notation and preliminaries}\label{s:notation}

Throughout the paper,  all neighborhoods are assumed to be open and connected and we denote, for any power series $u(x)$ with complex coefficients (centered at the origin),  by  $\bar u (x)$ the power series  obtained by taking complex conjugates of the coefficients of $u$.

Let $M\subset \C^N$ be real-analytic generic submanifold through the origin, of CR dimension $n$ and codimension $d$, so that $N=n+d$. We may assume that  $M=\{Z\in U:\rho (Z,\bar Z)=0\}$  where  $\rho=(\rho_1,\ldots, \rho_{d})$ is a real-analytic vector-valued defining function for $M$ defined on some neighborhood $U\subset \C^N$ of $0$ satisfying $\partial \rho_1\wedge \ldots \wedge \partial \rho_d\not =0$ over $U$. Choosing $U$ so that $U=\overline{U}$, we define the complexification of $M$ by 
$$\6M :=\{(Z,\zeta)\in U\times U: \rho (Z,\zeta)=0\}$$ as well as $\overline{\6M}=\{(\zeta,Z)\in U\times U: (Z,\zeta)\in \6M\}$. As in \cite{Z97}, we shall also consider the iterated complexifications $\6M^j$, for $j\geq 1$,  as follows. For $j=2\ell-1$ odd, we set 
$$\6M^{2\ell-1}:=
	\{(Z,\zeta^1,Z^1, \ldots,Z^{\ell-1},\zeta^{\ell})\in U\times \ldots \times U: 
	(Z,\zeta^1)\in \6M, 
	(\zeta^1, Z^1)\in \1{\6M}, 
	(Z^1,\zeta^2)\in \6M,
	\ldots, 
	(Z^{\ell-1},\zeta^{\ell})\in \6M\}$$
and for $j=2\ell$ even we set
$$\6M^{2\ell}:=\{
	(Z,\zeta^1,\ldots,Z^{\ell-1},\zeta^{\ell},Z^{\ell})\in U \times \ldots \times U: 
	(Z,\zeta^1)\in \6M, 
	(\zeta^1, Z^1)\in \1{\6M}, 
	(Z^1,\zeta^2)\in \6M,
	\ldots, 
	(\zeta^{\ell}, Z^{\ell})\in \1{\6M}
\}.$$
Recall that we can choose normal coordinates $Z=(z,w)\in \C^n \times \C^d$ for $M$ near $0$, so  that (the germ of) $M$ (at $0$) is given by 
\begin{equation}\label{e:define}
w=Q(z,\bar z,\bar w),
\end{equation}
where $Q=(Q^1,\ldots,Q^d)$ is a $\C^d$-valued  holomorphic map defined in some fixed neighborhood of the origin (see e.g. \cite{BERbook}).   Since $M$ is a real submanifold, the map $Q$ satisfies, in addition, the following identity:
\begin{equation}\label{e:real}
Q(z,\bar z,\bar Q (\bar z,,z,w))=w.
\end{equation}
Writing $\zeta =(\chi,\tau)\in \C^n \times \C^d$, let us define the following tangent vector fields to $\6M$, 
obtained from complexification of the $(0,1)$ vector fields on $M$:

\begin{equation}\label{e:Lj}
\6L_j:=\frac{\partial}{\partial \chi_j}+\sum_{\nu=1}^d \bar Q_{\chi_j}^d(\chi,z,w) \frac{\partial }{\partial \tau_\nu},\quad j=1,\ldots,n.
\end{equation}
We will make use the Segre maps associated to $M$, as introduced in \cite{BERacta, BERbook}.  Shrinking $U$ if necessary, for $p\in U$ (which later will furthermore lie on $M$), let us recall how are defined the Segre maps of order $\kappa \in \Z_+$. Following the notion of \cite{BRZ01}, we first set $v^0(p):=p$ and
\begin{equation}\label{e:mid251}
v^{\kappa+1}(t^0,t^1,\ldots, t^{\kappa };p)=(t^{0},Q(t^{0},\overline{v^\kappa}(t^1,\ldots,t^\kappa;p)).
\end{equation}
Note that the Segre maps are defined and holomorphic over $U_1\times \ldots \times U_1 \times U$ provided $U_1$ and $U$ are sufficiently small neighborhoods of the origin in $\C^n$ and $\C^N$ respectively. Since we will need only finitely many of those Segre maps,  we choose and fix neighborhoods $U_1$ and $U$  as above so that all these maps $v_\kappa$'s are well defined and holomorphic on $U_1^{\kappa+1}\times U$.

For every integer $\kappa \geq 1$, the real-analytic map $\Xi \colon U_1^{2\kappa}\times (M\cap U_1)\to \6M^{2\kappa +1}$  given by
\begin{equation}\label{e:param}
\Xi (t^0,\ldots, t^{2\kappa-1},p):=(v^{2\kappa}(t^0,\ldots,t^{2\kappa-1};p),\overline{v^{2\kappa-1}}(t^1,\ldots,t^{2\kappa-1};p),\ldots,\overline{v^1}(t^{2\kappa-1};p),p,\bar p)
\end{equation}
parametrizes the (germ at the origin of the) submanifold \begin{equation}\label{e:Nl}
\6N^\kappa=\{(Z,\zeta^1,\ldots,\zeta^\kappa,p,\bar p): (Z,\zeta^1,\ldots,\zeta^\kappa,p)\in \6M^{2\kappa},\ p\in M \}  \subset \6M^{2\kappa+1}.
\end{equation}

We recall the following finite type/minimality criterion from \cite{BER99}: 
\begin{Thm}\label{t:criterion}
Let $M$ be a germ of a real-analytic generic minimal submanifold at the origin. With the above notation, there exists an integer $s\le N$ such that the following holds:
\begin{equation}\label{e:rank}
{\rm max}\Big \{{\rm rk} \frac{\partial v^{2s}}{\partial (t^0,t^{s+1},t^{s+2},\ldots, t^{2s-1})}(0,x^1,\ldots,x^{s-1},x^s,x^{s-1},\ldots,x^1;0) :x^1,\ldots,x^s\in U_1 \Big \}=N
\end{equation}

\begin{equation}\label{e:vanishing}
v^{2s}(0,x^1,\ldots,x^{s-1},x^s,x^{s-1},\ldots,x^1;0) =0
\end{equation}

\end{Thm}

We shall also need another result from \cite{BER99} which can be seen as a version of the implicit function theorem with singularities (see \cite[Proposition 4.1.18]{BER99}).

\begin{Pro}\label{pro:BER}
Let $u(x,t,y)$ be a $\C^k$-valued holomorphic map 
defined in a neighborhood of the origin $\C^{r_1}\times \C^{r_2}\times \C^k$. Assume that 
$$u(x,0,0)=0,\quad \Delta (x):={\rm det} \left ( \frac{\partial u}{\partial y}(x,0,0)\right ) \not \equiv 0.$$
Then there exists  a $\C^k$-valued holomorphic map $\Theta$
defined  in a neighborhood of $0$ in $\C^{r_1+r_2+k}$, vanishing at $0$, such that 
$$u\left (x,t,\Delta (x) \, \Theta \left (x,\frac{t}{\Delta (x)^2},\frac{\sigma}{\Delta (x)^2}\right )\right)=\sigma$$
for all $(x,t,\sigma)\in \C^{r_1+r_2+k}$ such that $\Delta (x)\not =0$ and $x$ and $\left |\frac{t}{\Delta (x)^2}\right |+\left |\frac{\sigma}{\Delta (x)^2}\right |$ sufficiently small.
\end{Pro}

\section{Universal meromorphic parametrization of CR maps}\label{s:meromorphicparam}
The goal of this section is to prove a very precise universal meromorphic parametrization property for germs of  CR maps from generic real-analytic CR submanifolds into spheres. The exact statement is provided by Proposition
\ref{p:leader} below.  We will divide the proof of such a proposition into two steps. The first step involves the use of reflection type methods combined with  ideas from \cite{Z99} and \cite{LM17}. It aims at obtaining a universal meromorphic identity for germs of CR maps  on the iterated complexication $\6M^2$   (Proposition \ref{pro:step1}). Then, in the second step which is 
more in the spirit of \cite{BER99, BRZ01}, we iterate  such an identity on the iterated complexifications $\6M^\kappa$ for large enough $\kappa$, and use the minimality criterion Theorem \ref{t:criterion} together with Proposition \ref{pro:BER} to lift the meromorphic identity from the iterated complexication to the ambient space $\C^N$. This step requires careful analysis as our goal will be to reach a (meromorphic) parametrization property for germs of CR maps indicating the dependence on the base point where each germ is defined.

\subsection{Reflection}\label{ss:reflection}

We use the notation previously introduced in \S \ref{s:notation}.  We have the following result.

\begin{Pro}\label{pro:step1}
Let $M\subset \C^N$ be a germ of a generic real-analytic submanifold  at the origin. Then, shrinking the neighborhood $U$ if necessary, there exists a $\C^r$-valued holomorphic map $A(Z,\zeta^1,Z^1)$, for some integer $r\geq 1$, depending only on $M$, defined  on $U \times U \times U$, and two finite collections of (universal) holomorphic polynomial maps $P_1,\ldots,P_J$, $D_1,\ldots,D_J$, such that for every germ of a holomorphic map $f\colon (\C^N,0)\to \C^{N'}$ with $f(M)\subset \4S^{2N'-1}$, there exists  $1\leq j_0\leq J$, such that  
\begin{equation}\label{e:double}
f(Z)= \frac{
	P_{j_0}(A(Z,\zeta^1,Z^1), (\partial^\mu \bar f(\zeta^1), \partial^\mu f(Z^1))_{|\mu|\leq N'})
}{
	D_{j_0}(A(Z,\zeta^1,Z^1), (\partial^\mu \bar f(\zeta^1), \partial^\mu f(Z^1))_{|\mu|\leq N'})
}
\end{equation}
and 
$
	D_{j_0}(
		A(Z,\zeta^1,Z^1), (\partial^\mu \bar f(\zeta^1), \partial^\mu f(Z^1))_{|\mu|\leq N'}
	)\not \equiv 0
$ for all $(Z,\zeta^1,Z^1)\in \6M^2$ sufficiently close to the origin.
\end{Pro}

\begin{proof}
We start with the basic equation
\begin{equation}\label{e:basic}
\sum_{i=1}^{N'}|f_i|^2=1,
\end{equation}
that holds on $M$ sufficiently close to the origin.
We complexify  it to obtain on $\6M$:

\begin{equation}\label{e:basicplus}
\sum_{i=1}^{N'}f_i(Z)\bar f_i(\zeta)=1
\end{equation}
Applying combinations of vector fields from \eqref{e:Lj}, 
$\6L^{\alpha}=\6L_1^{\alpha_1}\ldots \6L_n^{\alpha_n}$ 
with $\alpha=(\alpha_1,\ldots,\alpha_n)$ 
and $|\alpha|\leq N'$ to \eqref{e:basicplus}, 
we obtain that for $(Z,\zeta)\in \6M$ sufficiently close to the origin:

\begin{equation}\label{e:hire}
\sum_{i=1}^{N'}f_i(Z)\6L^\alpha \bar f_i(\zeta)=0.
\end{equation}

For every $\alpha$ as above, we view $\6L^\alpha \bar f(\zeta)$ as a vector in $\C^{N'}$ (depending on $(Z,\zeta)\in \6M$). 
For $0\leq r\leq N'$, denote
by $e_r\le N'$ 
the generic rank (over  a sufficiently small neighborhood of $0$ in $\6M$) of the collection  of vectors $\6L^\alpha \bar f(\zeta)$ for $|\alpha|\leq r$. We clearly have that the sequence $e_r$, $0\leq r\leq N'$ strictly increases until it stabilizes (see e.g. \cite{La01}). Let $r_0\in \{1,\ldots,N'\}$ be defined by $e_{r_0-1}<e_{r_0}=e_{r_0+1}$ and set $k_0:=e_{r_0}$.  Even though $k_0$ depends on  the map $f$,  note that we have only finitely many choices for such an integer. 

In what follows, we assume that the so-called {\em generic degeneracy} $m:=N'-k_0>0$ (see \cite{LM17}), the simpler case $k_0=N'$ will be discussed at the end of the proof. In order to add some further equations to the system \eqref{e:hire}, we shall use arguments from \cite{BX15, LM17}.

It follows from \cite{BX15} or \cite[Proposition 4.4]{LM17} that there exists meromorphic maps 
$V^j \colon U\to \C^{N'}$, $V^j=(V^j_1, \ldots, V^j_{N'})$, $j=1,\ldots,m$, satisfying 
\begin{equation}\label{e:basket}
\sum_{i=1}^{N'}V_i^j(Z) \bar f_i(\zeta)=0
\end{equation}
for $(Z,\zeta)\in \6M \cap (U\times U)$, and such that
the matrix
$(V^1,\ldots,V^m)$ 
is of generic rank $m$. In fact, more can be said about how those maps $V^j$'s may be constructed. We explain this following the lines of \cite[Proposition 4.4]{LM17}. 

We choose $k_0$ multi-indices $\alpha^{(1)},\ldots, \alpha^{(k_0)}$ of length $\leq k_0$, with $\alpha^{(1)}=0$, such that the  generic rank of the matrix $\left (\6L^{\alpha^{(\ell)}} \bar f_j (\zeta)\right )_{1\leq j\leq N'\atop 1\leq \ell\leq k_0}$ equals $k_0$. 
Picking a generically invertible minor of size $k_0$ in this matrix, say the first minor on the top left of the matrix, we may write the desired map $V^j=(V_1^j,\ldots,V^j_{N'})$ in the form, 
$$V^j_i(Z)=\frac{P_{ij}\left (\left (\6L^{\alpha^{(\ell)}} \bar f (\zeta)\right )_{1\leq \ell  \leq k_0}\right)}{{\rm det}\left( \left (\6L^{\alpha^{(\ell)}} \bar f_j (\zeta)\right )_{1\leq j\leq k_0\atop 1\leq \ell\leq k_0}\right)},\quad  (Z,\zeta)\in \6M,$$ 
for some universal polynomials $P_{ij}$. 
Furthermore, as in \cite[Proposition 4.4]{LM17}, we have 
\begin{equation}\label{e:referee}
V_i^j(Z)=\delta_{i,k_0+j},\quad  i\geq k_0+1,\ 1\leq j\leq m,
\end{equation}
where $\delta_{i,k_0+j}$ denotes the usual Kronecker symbol. In particular, the generic rank of  
the matrix
$(V^1,\ldots,V^m)$ 
is equal to $m$. Since there are finitely many choices for the above mentioned minors, as well as for the multi-indices $\alpha^{\ell}$'s, we therefore come to the conclusion that we may write for each $j=1,\ldots,m$, 
\begin{equation}\label{e:switch}
V^j(Z)=\frac{P_j \left (\left (\6L^{\alpha} \bar f (\zeta)\right )_{|\alpha|\leq N'}\right)}{D \left (\left (\6L^{\alpha} \bar f (\zeta)\right )_{|\alpha|\leq N'}\right)},\quad (Z,\zeta)\in \6M, 
\end{equation}
where $P_j$ and $D$ belong to a finite family of universal polynomial maps (with real coefficients) and $D \left (\left (\6L^{\alpha} \bar f (\zeta)\right )_{|\alpha|\leq N'}\right)\not \equiv 0$ for $(Z,\zeta)\in \6M$. Now we note that for $j=1,\ldots, m$, \eqref{e:switch} may be rewritten as follows

\begin{equation}\label{e:switchbis}
\overline V^j(\zeta)=\frac{P_j \left (\left (\6T^{\alpha}  f (Z^1)\right )_{|\alpha|\leq N'}\right)}{D \left (\left (\6T^{\alpha}  f (Z^1)\right )_{|\alpha|\leq N'}\right)},
\quad 
(\zeta, Z^1)\in \1{\6M}, 
\end{equation}
where we write $Z^1=(z^1,w^1)\in \C^n \times \C^d$ and 
\begin{equation}\label{e:Tj}
\6T_r:=\frac{\partial}{\partial z^1_r}+\sum_{\nu=1}^d Q_{z_r}^d(z^1,\chi,\tau) \frac{\partial }{\partial w^1_\nu},\quad r=1,\ldots,n.
\end{equation}

Conjugating \eqref{e:basket} and adding it to the system of 
\eqref{e:basicplus} and \eqref{e:hire}, 
we obtain the following system of equations on $\6M$ (with meromorphic coefficients):

\begin{equation}\label{e:system}
\begin{cases}
\sum_{i=1}^{N'}f_i(Z)\, \6L^{\alpha^{(1)}} \bar f_i(\zeta)&=1\\
\sum_{i=1}^{N'}f_i(Z)\, \6L^{\alpha^{(\ell)}} \bar f_i(\zeta)&=0,\ 2\leq \ell \leq k_0,\\
\sum_{i=1}^{N'}\overline{V}_i^j(\zeta)f_i(Z)&=0,\ j=1,\ldots,m.
\end{cases}
\end{equation}
We now claim the following:

{\bf Claim.} {\em The matrix $\6B(Z,\zeta)$ formed with the column vectors $\6L^{\alpha^{(1)}}\bar f (\zeta),\ldots, \6L^{\alpha^{(k_0)}}\bar f (\zeta),\overline{V}^1(\zeta),\ldots,\overline{V}^m(\zeta)$ has generic rank $N'$ (over $\6M$).}

\medskip

Let us prove the claim following the arguments of \cite[Theorem 5.2]{LM17}. Using \eqref{e:basket} and \eqref{e:referee}, we have for $(Z,\zeta)\in \6M$ near the origin
\begin{equation}
\bar f_{k_0+j}(\zeta)=-\sum_{i=1}^{k_0}V_{i}^j(Z)\bar f_i(\zeta),\quad 1\leq j\leq m,
\end{equation}
and hence
\begin{equation}\label{e:livre}
\6L^{\alpha^{(\ell)}}\bar f_{k_0+j}(\zeta)=-\sum_{i=1}^{k_0}V_i^j(Z)\, \6L^{\alpha^{(\ell)}}\bar f_i(\zeta), \quad 1\leq j\leq m,\quad 1\leq \ell \leq k_0.
\end{equation}
For every  $1\leq \nu \leq N'$,  denote by $R_\nu$ the $\nu$-th row of the matrix $\6B$.  Substituting, for every such $k_0+1\leq \nu \leq N'$, $R_\nu$ by $R_\nu+\sum_{i=1}^{k_0}V_i^{\nu-k_0}(Z)R_i$, we obtain, in view of \eqref{e:livre} a matrix of the form

\[
\sbox0{$\begin{matrix}\6L^{\alpha^{(1)}}\bar f_1(\zeta)&\ldots&\6L^{\alpha^{(k_0)}}\bar f_1(\zeta)\\ \vdots &&\vdots \\
 \6L^{\alpha^{(1)}}\bar f_{k_0}(\zeta)& \ldots &  \6L^{\alpha^{(k_0)}}\bar f_{k_0}(\zeta) \end{matrix}$}
\sbox1{$\begin{matrix} \overline{V}^1_1(\zeta)&\ldots&\overline{V}_1^m(\zeta)\\ \vdots&\ldots&\vdots\\   \overline{V}^1_{k_0}(\zeta) &\ldots& \overline{V}^m_{k_0}(\zeta)\end{matrix}$}
\left[
\begin{array}{c|c}
\usebox{0}&\usebox{1}\\
\hline
  \vphantom{\usebox{0}}\makebox[\wd0]{\large $0$}&\makebox[\wd0]{\large $C(Z,\zeta)$}
\end{array}
\right],
\]
where $C(Z,\zeta)$ is the $m\times m$ matrix given by $\left(V^k(Z)\cdot \overline{V}^\nu (\zeta)\right)_{k,\nu}$ and $V^k(Z)\cdot \overline{V}^\nu (\zeta)=\sum_{i=1}^{N'}V^k(Z)\overline{V}^\nu(\zeta)$. Hence $C(Z,\zeta)$ is simply the complexification of the Gram matrix $\left(V^k\cdot \overline{V}^\nu\big|_{M}\right)_{k,\nu}$, which is generically invertible since the vectors $V^k\big|_M$, $k=1,\ldots,m$, are generically linearly independent (near $0$). The proof of the claim is complete.

Thanks to the claim, we may now finish the proof of the proposition and  solve the linear system of equations \eqref{e:system} (in the $f$'s) using Cramer's rule and obtain that for $(Z,\zeta)\in \6M$
\begin{equation}\label{e:nil}
f(Z)=\frac{\widetilde P\left (\left (\6L^{\alpha} \bar f (\zeta)\right )_{|\alpha|\leq N'}, \overline{V}(\zeta) \right )}{\widetilde D \left (\left (\6L^{\alpha} \bar f (\zeta)\right )_{|\alpha|\leq N'},\overline{V}(\zeta)\right)}
\end{equation}
where $\widetilde P$ and $\widetilde D$ are, respectively, universal polynomial $\C^{N'}$-valued and $\C$-valued maps, depending on the map $f$, but belonging to a {\em finite} collection of universal polynomial maps. Now substituting \eqref{e:switchbis} into \eqref{e:nil} yields that we may write 

\begin{equation}\label{e:nilbis}
f(Z)=\frac{\widehat P\left (\left (\6L^{\alpha} \bar f (\zeta)\right )_{|\alpha|\leq N'}, \left (\6T^{\alpha}  f (Z^1)\right )_{|\alpha|\leq N'} \right )}{\widehat D  \left (\left (\6L^{\alpha} \bar f (\zeta)\right )_{|\alpha|\leq N'}, \left (\6T^{\alpha}  f (Z^1)\right )_{|\alpha|\leq N'} \right)}
\end{equation}
for $(Z^1,\zeta,Z)\in \6M^2$ sufficiently close to the origin and where $\widehat P$ and $\widehat D$ belonging to some finite collection of universal polynomial maps. Using the form of the vector fields $\6L_j$'s and $\6T_r$'s given by \eqref{e:Lj} and \eqref{e:Tj},  we get the desired statement of the proposition.

To complete the proof of the proposition, we must tackle the case where $k_0=N'$. In that case, we can directly apply Cramer's rule to the system of equations given by \eqref{e:basicplus} and \eqref{e:hire} and reach a similar conclusion as the one obtained when $k_0<N'$. We leave the details to the reader. The proof of the proposition is complete.
\end{proof}

\begin{Rem} 
\begin{enumerate}
\item[{\rm (a)}] We note that the map $A$ in the right hand side of \eqref{e:double} is defined in the fixed neighborhood $U\times U\times U$ of $0$ where $U$ is given as in \S \ref{s:notation}.
\item[{\rm (b)}] In Proposition \ref{pro:step1} as well as in further propositions below, we obtain the existence of universal polynomial maps satisfying certain properties. Universality means that the polynomial maps are independent of the given manifold $M$ (and of any chosen point and neighborhood there) as well as independent of all the germs of CR maps under consideration.
\end{enumerate}
\end{Rem}

\subsection{Iteration}\label{ss:iteration}

Our next goal is to get  a similar identity as the one in \eqref{e:double}, but on the iterated complexification of any order instead.

Differentiating \eqref{e:double} and using the chain rule, one  easily gets the following statement:

\begin{Pro}\label{e:derivatives}
Let $M$ and $U$ be as above, $A$ and $P_j,D_j$, $1\leq j\leq J$ be given by Proposition \ref{pro:step1}. Then for every multi-index $\gamma \in \N^N$, there exists a holomorphic map $A^{(\gamma)}(Z,\zeta^1,Z^1)$, depending only on $M$, defined on $U\times U\times U$, and for every $1\leq j\leq J$,  a (universal) holomorphic polynomial map $P^{(\gamma)}_j$, depending only on  $P_j$, $D_j$ and $\gamma$, such that  for every germ of a holomorphic map $f\colon (\C^N,0)\to \C^{N'}$ with $f(M)\subset \4S^{2N'-1}$, there exists  $1\leq j_0\leq J$, such that for every $\gamma \in \N^N$, 
\begin{equation}\label{e:doublederivative}
(\partial^\gamma f)(Z)= \frac{
	P_{j_0}^{(\gamma)}
		\left(
			A^{(\gamma)}(Z,\zeta^1,Z^1), 
			(\partial^\mu \bar f(\zeta^1), \partial^\mu f(Z^1))_{|\mu|\leq N'+|\gamma|}
		\right)}{
	\left (
		D_{j_0} \left (
			A(Z,\zeta^1,Z^1), 
			(\partial^\mu \bar f(\zeta^1), 
			\partial^\mu f(Z^1))_{|\mu|\leq N'}
		\right )\right )^{2|\gamma|}
},
\end{equation}
for all $(Z,\zeta^1,Z^1)\in \6M^2$ sufficiently close to the origin.
\end{Pro}

Now iterating Proposition \ref{e:derivatives} along the iterated complexification yields the following statement.

\begin{Pro}\label{pro:iteratedcomplex}
Let $M$ and $U$ be as above and let $P_j,D_j$, $1\leq j\leq J$, be given by Proposition \ref{pro:step1}. Fix an integer $\ell \geq 1$. Then there exists a holomorphic map $A_\ell (Z,\zeta^1,Z^1,\ldots,Z^\ell,\zeta^{\ell+1})$, depending only on $M$, defined  on $U^{2\ell+2}$, and for every $0\leq j\leq J$,  (universal) holomorphic polynomial maps $P_{j,\ell}$ and $D_{j,\ell}$, depending only on $P_j$, $D_j$ and $\ell$, such that  for every germ of a holomorphic map $f\colon (\C^N,0)\to \C^{N'}$ with $f(M)\subset \4S^{2N'-1}$, there exists  $1\leq j_0\leq J$, 
\begin{equation}\label{e:triplederivative}
f(Z)= \frac{P_{j_0,\ell} \left (A_\ell (Z,\zeta^1,Z^1,\ldots, Z^\ell,\zeta^{\ell+1}), (\partial^\mu \bar f(\zeta^{\ell+1}), \partial^\mu f(Z^{\ell}))_{|\mu|\leq 2\ell N'} \right )}{
	D_{j_0,\ell}
		\left (
			A_\ell	
				(Z,\zeta^1,Z^1,\ldots, Z^\ell, \zeta^{\ell+1}), 
				(\partial^\mu \bar f(\zeta^{\ell+1}), \partial^\mu f(Z^{\ell}))_{|\mu| \leq 2\ell N'}
		\right )},
\end{equation}
for all $(Z,\zeta^1,Z^1,\ldots, Z^\ell, \zeta^{\ell+1})\in \6M^{2\ell+1}$ sufficiently close to the origin, and the denominator in \eqref{e:triplederivative} does not vanish identically on $\6M^{2\ell+1}$. In particular, we have the following representation:
\begin{equation}\label{e:elect}
f(Z)= \frac{P_{j_0,\ell} \left (
	A_\ell (Z,\zeta^1,Z^1,\ldots, Z^{\ell -1}, \zeta^{\ell}, p,\bar p), (\partial^\mu \bar f(\bar p), \partial^\mu f(p))_{|\mu|\leq 2\ell N'} \right )}{
	D_{j_0,\ell}\left (
		A_\ell(Z,\zeta^1,Z^1,\ldots, Z^{\ell-1}, \zeta^{\ell},p,\bar p), 
		(\partial^\mu \bar f(\bar p), \partial^\mu f(p))_{|\mu| \leq 2 \ell N'}
	\right )},
\end{equation}
for all $(Z,\zeta^1,Z^1,\ldots, Z^{\ell-1}, \zeta^{\ell},p)\in \6N^{\ell}$ (given by \eqref{e:Nl}) sufficiently close to the origin, and the denominator in \eqref{e:elect} does not vanish identically (on $\6N^\ell$).
\end{Pro}

\begin{proof} The proof consists of a systematic use of Proposition \ref{e:derivatives}. 

Let $f$ be as in the proposition. Applying Proposition \ref{pro:step1}, we have for some $1\leq j_0\leq J$, 
\begin{equation}\label{e:weekend}
f(Z)= \frac{P_{j_0}(A(Z,\zeta^1,Z^1), (\partial^\mu \bar f(\zeta^1), \partial^\mu f(Z^1))_{|\mu|\leq N'})}{
	D_{j_0}(
		A(Z,\zeta^1,Z^1), (\partial^\mu \bar f(\zeta^1), \partial^\mu f(Z^1))_{|\mu|\leq N'}
	)},
\end{equation}
with a non-vanishing denominator on $\6M^2$. Now applying Proposition \ref{e:derivatives} and taking the complex conjugate of \eqref{e:doublederivative}, we have for every multi-index $\mu \in \N^N$, and for every $(Z^1,\zeta^2)\in \6M$ 
and 
$(\zeta^1, Z^1)\in \1{\6M}$
sufficiently close to the origin, 
\begin{equation}\label{e:barstuff}
(\partial^\mu \bar f)(\zeta^1)= \frac{\overline{P_{j_0}^{(\gamma)}}\left (\overline{A^{(\gamma)}}(\zeta^1,Z^1,\zeta^2), (\partial^\nu  f(Z^1), \partial^\nu \bar f(\zeta^2))_{|\nu|\leq N'+|\mu|} \right)}{
	\left (\overline{D_{j_0}} \left(
		\bar A(\zeta^1,Z^1,\zeta^2), (\partial^\nu  f(Z^1), \partial^\nu \bar f(\zeta^2))_{|\nu|\leq N'}
		\right) \right )^{2|\mu|}}.
\end{equation}
Substituting \eqref{e:barstuff} into \eqref{e:weekend} immediately gives the \eqref{e:triplederivative} for $\ell =1$. The general case of \eqref{e:triplederivative} for arbitrary $\ell$ follows from the same type of arguments. 

The last statement of the  proposition follows from the fact that  the submanifold $\6N^\ell$ is a uniqueness set for holomorphic functions on $\6M^{2\ell+1}$.  The proof of the proposition is complete now.
\end{proof}


Using the iterated Segre maps as recalled in \S \ref{s:notation}, we now reach the following. 


\begin{Pro}\label{pro:iteratedbis}
Let $M\subset \C^N$ be a generic real-analytic minimal submanifold through the origin and let $s\in \Z_+$ be as Theorem \ref{t:criterion}.  Then there exists a holomorphic map $\Phi (t^0,\ldots,t^{2s-1},\lambda, \omega)$, depending only on $M$, defined  on some fixed neighborhood of $0\in \C^{2sn}\times \C_\lambda^N \times \C_\omega^N$, and two finite collections of (universal) holomorphic polynomial maps $\widetilde P_1,\ldots,\widetilde P_J$, $\widetilde D_1,\ldots,\widetilde D_J$,  such that  for every germ of a holomorphic map $f\colon (\C^N,0)\to \C^{N'}$ with $f(M)\subset \4S^{2N'-1}$, there exists  $j_0\in \{1,\ldots,J\}$ such that 
\begin{equation}\label{e:runaway}
(f\circ v^{2s})(t^0,\ldots,t^{2s-1};p)= \frac{\widetilde P_{j_0}\left (\Phi(t^0,\ldots,t^{2s-1},p,\bar p), (\partial^\mu \bar f(\bar p), \partial^\mu f(p))_{|\mu|\leq 2sN'} \right )}{
	\widetilde D_{j_0}\left (
		\Phi (t^0,\ldots,t^{2s-1},p,\bar p), 
		(\partial^\mu \bar f(\bar p), \partial^\mu f(p))_{|\mu| \leq 2sN'}
	 \right )},
\end{equation}
for all  $(t^0,\ldots,t^{2s-1})\in \C^{2sn}$ and $p\in M$ sufficiently close to $0$, with  the denominator in \eqref{e:runaway}  not vanishing identically. 
\end{Pro}

\begin{proof} Setting $\Phi (t^0,\ldots,t^{2s-1},p,\bar p)=(A_s \circ \Xi)(t^0,\ldots, t^{2s-1},p)$ with $A_s$ and $\Xi$ given respectively by Proposition \ref{pro:iteratedcomplex}
 and \eqref{e:param}, we see that the result follows from Proposition \ref{pro:iteratedcomplex}.
\end{proof}

\subsection{Lifting}

The next step consists of the lifting procedure. In order to carry it out, we need a more precise version of Proposition \ref{pro:iteratedbis} since we will be considering germs of CR maps at points $q\in M$ sufficiently close to the origin. To this end, we'll be more precise in the choice of our neighborhoods.  Let $s\in \Z_+$ be the integer given as before. We assume that $M$ is given by \eqref{e:define} for $|Z|<C_0$ for some fixed $C_0>0$. We also assume that the iterated Segre map $v^{2s}(t^0,\ldots,t^{2s-1};p)$ is defined for $|t^j|<C_1$, $|p|<C_1$, $j=0,\ldots,2s-1$, with $0<C_1\leq C_0$. Now inspecting the proofs in \S \ref{ss:reflection} and \S \ref{ss:iteration}, we have the following:

\begin{Pro}\label{pro:inspection}
Let $M\subset \C^N$ be a generic real-analytic minimal submanifold through the origin. 
Then for $C_1>0$ sufficiently small, there exists a holomorphic map $\Phi (t^0,\ldots,t^{2s-1},\lambda, \omega)$, depending only on $M$, defined  for $t^j\in \C^n$, $\lambda \in \C^N$,  $\omega \in \C^N$,   $|t^j|<C_1$, $|\lambda|<C_1$, $|\omega|<C_1$, $j=0,\ldots,2s-1$, and two finite collections of (universal) holomorphic polynomial maps $\widetilde P_1,\ldots,\widetilde P_J$, $\widetilde D_1,\ldots,\widetilde D_J$,
such that  if 
$q=(z_q,w_q)\in M$ with $|q|<C_1$ and
  $f\colon (\C^N,q)\to \C^{N'}$ is a germ  of a holomorphic map with $f(M)\subset \4S^{2N'-1}$, 
  there exists $j_0\in \{1,\ldots,J\}$ such that
\begin{equation}\label{e:runawaybis}
(f\circ v^{2s})(t^0,\ldots,t^{2s-1};p)= \frac{\widetilde P_{j_0}\left (\Phi(t^0,\ldots,t^{2s-1},p,\bar p), (\partial^\mu \bar f(\bar p), \partial^\mu f(p))_{|\mu|\leq 2sN'} \right )}{
\widetilde D_{j_0}\left (\Phi (t^0,\ldots,t^{2s-1},p,\bar p), (\partial^\mu \bar f(\bar p), \partial^\mu f(p))_{|\mu| \leq 2sN'} \right )},
\end{equation}
for all  $(t^0,\ldots,t^{2s-1})\in \C^{2sn}$ and $p\in M$ sufficiently close to $(z_q,\bar z_q,\ldots, z_q,\bar z_q)$ and $q$ respectively, with the denominator in \eqref{e:runawaybis}  not vanishing identically. \end{Pro}



We now want to lift \eqref{e:runawaybis} to get a universal meromorphic parametrization property for CR maps as given by Proposition \ref{p:leader} below. The proof consists of a careful application of Proposition \ref{pro:BER} following in spirit some steps from \cite{BER99, BRZ01}.

In what follows, we assume that $s$ is odd, the even case, being very similar, is left  to the reader.

We introduce the following variables 
$$
	x=(x^1,\ldots,x^s)\in \C^{ns},
	\quad
	u=(u^0,\ldots,u^{s-1}) \in \C^{sn},
	\quad
	(\eta,\sigma)\in \C^n \times \C^d,
	\quad 
	\theta \in \C^n, 
	\quad
	\omega \in \C^N,
$$
 and define  holomorphic maps $L\colon (\C^{2(s+1)n},0)\to (\C^{2sn},0)$ and  $\vartheta \colon (\C^{2(s+1)n+N+d},0)\to (\C^N,0)$ as follows:

$$L(x,u,\eta,\theta):=(u^0+\eta, x^1+\theta,x^2+\eta,\ldots, x^s+\theta, u^{s-1}+x^{s-1}+\eta,\ldots, u^1+x^{1}+\theta),$$
$$\vartheta (x,u,\eta,\sigma,\theta, \omega)=v^{2s} \left (L(x,u,\eta,\theta);\omega+(\eta,\sigma) \right )-(\eta,\sigma).$$
 We choose $0<C_2<C_1$ such that $\vartheta$ is a holomorphic  map for
\begin{equation}\label{e:domain}
|x|<C_2,\, |u|<C_2, |\eta|<C_2, |\sigma|<C_2, |\omega|<C_2, |(\eta,\sigma)|<C_2, |\theta|<C_2.
\end{equation}
We also define
$$\Psi (x,u,\eta,\sigma,\theta,\omega):=\Phi \left (L(x,u,\eta,\theta);\omega+(\eta,\sigma),\overline{\omega+(\eta,\sigma)}\right ),$$
where $\Phi$ is given by Proposition \ref{pro:inspection}. 
Choosing $C_2>0$ sufficiently small, we may assume that  $\Psi$ is real-analytic on the open set given by \eqref{e:domain} and holomorphic with respect to $x,u,\theta$.

In view of \eqref{e:vanishing}, we have 
\begin{equation}\label{e:condition1}
\vartheta (x,0,0,0,0)=0,
\end{equation}
and in view of \eqref{e:rank}, we also have
\begin{equation}\label{e:rankbis}
{\rm max}\Big \{{\rm rk} \frac{\partial \vartheta}{\partial u}(x,0,0,0,0) : |x|<C_2 \Big \}=N.
\end{equation}

We write $u=(\xi,y)\in \C^{sn-N}\times \C^N$ such that 

\begin{equation}\label{e:stuff}
\Delta (x)={\rm det} \left ( \frac{\partial \vartheta}{\partial y}(x,0,0,0,0)\right)\not \equiv 0.
\end{equation}
By Proposition \ref{pro:BER}, there exists a $\C^N$-valued holomorphic map $\Upsilon (x,\xi,\eta,\sigma,\theta, \omega)$ defined for 
\begin{equation}\label{e:getit}
|x|,|\xi|,|\eta |,|\sigma|, |\theta|, |\omega|<C_3,
\end{equation}
 for some $0<C_3<C_2$, such that 
\begin{equation}\label{e:benefits}
\vartheta \left (x,\xi, \Delta (x)\, \Upsilon \left ( x, \frac{\xi}{\Delta (x)^2}, \frac{(\eta,\sigma)}{\Delta (x)^2}\frac{\theta}{\Delta (x)^2}, \frac{\omega}{\Delta (x)^2},\frac{Z^0}{\Delta (x)^2}\right ),\eta,\sigma, \theta, \omega \right )=Z^0
\end{equation}
for all $(x,\xi,\eta, \sigma,\theta,\omega,Z^0)$ belonging to the open subset $\6W$  where  \eqref{e:getit} holds,  $\Delta (x)\not =0$ and 
\begin{equation}
\left | \frac{\xi}{\Delta (x)^2} \right |, \left |\frac{(\eta,\sigma)}{\Delta (x)^2}\right |, \left |\frac{\theta}{\Delta (x)^2}\right |, \left | \frac{\omega}{\Delta (x)^2} \right |, \left |\frac{Z^0}{\Delta (x)^2}\right | < C_4
\end{equation}
for some constant $0<C_4<C_3$ (chosen in such a way that the term on the left of \eqref{e:benefits} is holomorphic on $\6W$). Reducing $C_3$ and $C_4$ further if necessary, we may assume that the map 
\begin{equation}\label{e:rift}
T(x,\xi,\eta,\sigma, \theta,\omega,Z^0):=\Psi  \left (x,\xi,  \Delta (x)\, \Upsilon \left ( x, \frac{\xi}{\Delta (x)^2}, \frac{(\eta,\sigma)}{\Delta (x)^2}\frac{\theta}{\Delta (x)^2}, \frac{\omega}{\Delta (x)^2},\frac{Z^0}{\Delta (x)^2}\right ) ,\eta,\sigma, \theta, \omega \right )
\end{equation}
is real-analytic on $\6W$ and holomorphic with respect to $(x,\xi,\theta,Z^0)$.

Pick an arbitrary relatively compact open subset $S$ of $\{x\in \C^{sn}:|x|<C_3,\ \Delta (x)\not =0\}$ and set $\delta:={\rm inf}\{\Delta (x):x\in S\}>0$. Reducing $C_4$ if necessary, we may assume that $\delta^2 C_4<C_3$.

Let $\Omega:= \left \{Z\in \C^N:|Z|< \frac{\delta^2 C_4}{2}\} \right \}$ and let  $q=(z_q,w_q)\in \Omega \cap M$ be arbitrary. Using what we have done before with $(\eta,\sigma)=(z_q,w_q)=q$ and $\theta =\bar \eta=\bar z_q$, we see that

\begin{equation}\label{e:gone}
\vartheta \left (x,\xi, \Delta (x)\, \Upsilon \left ( x, \frac{\xi}{\Delta (x)^2}, \frac{q}{\Delta (x)^2}\frac{\bar z_q}{\Delta (x)^2}, \frac{\omega}{\Delta (x)^2},\frac{Z^0}{\Delta (x)^2}\right ),q, \bar z_q, \omega \right )=Z^0
\end{equation}
for all $x\in S$, $|\xi |<\delta^2 C_4$, $|\omega |<\delta^2C_4$, and $|Z^0|<\delta^2 C_4$. Using  \eqref{e:real}, \eqref{e:mid251} and the fact that $q\in M$, we further notice that 
$$\vartheta (0,0,q,\bar z_q,0)=v^{2s}\left ( L(0,0,z_q,\bar z_q);q\right)-q=v^{2s}(z_q,\bar z_q,\ldots,\bar z_q;q)-q=0.$$

Consider now a germ of a holomorphic map $f\colon (\C^N,q)\to \C^{N'}$, sending $(M,q)$ into $\4S^{2N'-1}$. Using Proposition \ref{pro:inspection} and its notation, writing $P=\widetilde P_{j_0}$, $D=\widetilde D_{j_0}$ and using \eqref{e:runawaybis},  we have
\begin{equation}\label{e:plaire}
f(q+\vartheta (x,u,q,\bar z_q,\omega))=\frac{P\left (\Psi (x,u,q,\bar z_q,\omega)
, (\partial^\mu \bar f(\overline{\omega +q}), \partial^\mu f(\omega+q))_{|\mu|\leq 2sN'} \right )}{D\left (\Psi (x,u,q,\bar z_q,\omega), (\partial^\mu \bar f(\overline{\omega+q}), \partial^\mu f(\omega+q))_{|\mu| \leq 2sN'}
 \right )}
\end{equation}
for all $(x,u)\in \C^{2ns}$ in a sufficiently small neighborhood of $0$ (depending on $q$) and for all $\omega$ in a sufficiently small neighborhood, denoted $M^{(q)}$, of  the origin in $\C^N \cap \{\omega: \omega+q\in M\}$. Furthermore  the denominator in \eqref{e:plaire} does not vanishing identically for all the above $(x,u,\omega)$'s since the linear  map $(x,u)\mapsto L(x,u,z_q,\bar z_q)$ is  invertible.

Next we observe that the right hand side of \eqref{e:plaire} is well-defined, as a ratio, for $|x|,|u|<C_3$ and for $\omega \in M^{(q)}$. We now claim that the left-hand side is also defined and holomorphic for $|x|<C_3$ and for $|u|$ sufficiently small (depending on $q$) and for $\omega \in M^{(q)}$. This claim follows from the fact that for every $x$ with $|x|<C_3$,
\begin{equation}
\begin{aligned}
\vartheta (x,0,q,\bar z_q,0)&=v^{2s}(L(x,0,z_q,\bar z_q);q)-q\\
&=v^{2s}(z_q,x^1+\bar z_q,\ldots,x^s+\bar z_q,x^{s-1}+z_q,\ldots,x^1+\bar z_q;q)-q\\
&=q-q=0,
\end{aligned}
\end{equation}
which  itself follows from \eqref{e:real} and the fact that $q\in M$. All this implies that the equality \eqref{e:plaire} holds for all $|x|<C_3$ (which is independent of the mapping $f$ and $q$) and for all $u,\omega$ sufficiently small (depending on $f$ and $q$). Now we may use \eqref{e:gone} to get the following identity
\begin{equation}\label{e:graal}
f(q+Z^0)=\frac{P\left (T(x,\xi,q,\bar z_q,\omega,Z^0)
, (\partial^\mu \bar f(\overline{\omega +q}), \partial^\mu f(\omega+q))_{|\mu|\leq 2sN'} \right )}{D\left (T (x,\xi,q,\bar z_q,\omega,Z^0), (\partial^\mu \bar f(\overline{\omega+q}), \partial^\mu f(\omega+q))_{|\mu| \leq 2sN'} 
\right )}
\end{equation}
for all $x\in S$ and $\xi,Z^0,\omega$ sufficiently small (depending on $q$) and $\omega \in M^{(q)}$. Furthermore, the reader may easily check that  the  mapping $Z^0\mapsto \Delta (x) \Upsilon \left ( x, \frac{\xi}{\Delta (x)^2}, \frac{q}{\Delta (x)^2}\frac{\bar z_q}{\Delta (x)^2}, \frac{\omega}{\Delta (x)^2},\frac{Z^0}{\Delta (x)^2}\right )$ is of full rank $N$ for $(x,\xi,Z^0,\omega)$'s as above. This implies that the denominator  in \eqref{e:graal} does not vanish identically (for all above $(x,\xi,Z^0,\omega)'s$). We may rewrite \eqref{e:graal} as follows:
\begin{equation}\label{e:graalbis}
f(Z)=\frac{
	P\left (T(x,\xi,q,\bar z_q,\omega,Z-q)
, (\partial^\mu \bar f(\overline{\omega +q}), \partial^\mu f(\omega+q))_{|\mu|\leq 2sN'} \right )
}{
	D\left (
		T (x,\xi,q,\bar z_q,\omega,Z-q), (\partial^\mu \bar f(\overline{\omega+q}), \partial^\mu f(\omega+q))_{|\mu| \leq 2sN'}
	\right )
}
\end{equation}
for all $Z$ close to $q$ and $x,\xi,\omega$ as above. 

Set $t=(x,\xi)$ and $H(t,p,q,Z):=T(x,\xi,q,\bar z_q,p-q,Z-q)$. Then, by the above, $H$ is a real-analytic map for $x\in S$, $|\xi|<\delta^2C_4$, $|p|<\frac{\delta^2 C_4}{2}$, $|q|<\frac{\delta^2 C_4}{2}$ and $|Z|< \frac{\delta^2 C_4}{2}$, and holomorphic in $(t,Z)$. We thus have proved the following:

\begin{Pro}\label{p:leader} 
Let $M\subset \C^N$ be a generic real-analytic minimal submanifold through the origin. Then there exist  a real-analytic map $H(t,p,q,Z)$ defined on some open polydisc $V \times W^3 \subset \C^r\times \C^{3N} $ for some integer $r\geq 1$, holomorphic with respect to $(t,Z)$, with $0\in W$,  a finite collection of (universal) holomorphic  $\C^{N'}$-valued polynomial maps $P_1,\ldots,P_J$, and a finite collection of (universal) holomorphic polynomials $D_1,\ldots,D_J$, such that for every $q\in M\cap W$, and every germ of a holomorphic map $f\colon (\C^N,q)\to \C^{N'}$ with $f(M)\subset \4S^{2N'-1}$, there exists  $1\leq j_0\leq J$, such that  for every $p\in M$ and $Z\in \C^N$ sufficiently close to $q$,  and every $t\in V$, 
\begin{equation}\label{e:eleven}
f(Z)= \frac{
	P_{j_0}\left (
		H (t,p,q,Z), (\partial^\mu \bar f(\bar p), \partial^\mu f(p))_{|\mu|\leq 2sN'} 
	\right )
}{
	D_{j_0}\left (
		H (t,p,q,Z), (\partial^\mu \bar f(\bar p), \partial^\mu f(p))_{|\mu| \leq 2sN'}
	\right )
},
\end{equation}
where the denominator in \eqref{e:eleven} does not vanish identically for $(t,p,Z)$ as above.
\end{Pro}


\section{Proofs of Theorems \ref{t:mainbest} and  \ref{t:meromorphic} and Corollary \ref{c:global} }\label{s:final}

\subsection{Meromorphic extension to a larger neighborhood -- Proof of Theorem \ref{t:meromorphic}}
Without loss of generality, we may assume that $p_0=0$. Let $\Omega:=W$ where $W$ is given by Proposition \ref{p:leader}. Let $q\in \Omega$ and  $f\colon (M,q)\to \4S^{2N'-1}$ be a germ of a $\6C^\infty$-smooth CR map.  As already mentioned, we may assume that $f$ extends holomorphically to a neighborhood of $q$ in $\C^N$. Choosing some value of $t\in V$ and $p\in W$ such that the denominator in \eqref{e:eleven} does not vanish identically, we see that that $f$ admits a meromorphic extension to all of $\Omega$. The second part of the theorem follows from the first part in conjunction with \cite{Chia}. The proof is complete.

\subsection{Proof of Corollary \ref{c:global}} Fix $p_0\in M$ and $f\colon (M,p_0)\to \4S^{2N'-1}$ as  in the corollary. It follows from Theorem \ref{t:meromorphic} that $f$ extends holomorphically along any path in $M$ starting from $p_0$. Hence, since $M$ is connected and simply-connected, we can  extend, by analytic continuation, the local map $f$ holomorphically to a neighborhood of $M$ in $\C^N$.

The second part of Corollary \ref{c:global} follows from the first part of it together
with the regularity results in \cite{F89, Mi03}.

\subsection{Unique jet determination -- Proof of Theorem \ref{t:mainbest}}


The proof of Theorem \ref{t:mainbest} will follow once we have proved the following:


\begin{Pro}\label{p:dobetter}
Let $M\subset \C^N$ be a generic real-analytic minimal submanifold through $0$. 
Then there exists a neighbhorhood $U_0$ of $0$ in $\C^N$ and an integer $K>0$ such that for every $q\in M\cap U_0$, if $f,g\colon (\C^N,q)\to \C^{N'}$ are two germs of holomorphic maps sending $M$ into $\4S^{2N'-1}$ with $j_q^Kf=j_q^Kg$, then $f=g$.
\end{Pro}

\begin{proof}

Let $H$, $V$, $W$  and the collection of polynomial maps $P_j$ and $D_j$, $1\leq j\leq J$, be  given by Proposition \ref{p:leader}.  For each $\C^{N'}$-valued polynomial map $P_j$, we write $P_j=(P_{j,1},\ldots,P_{j,N'})$. Shrinking $V$ and $W$ if  necessary we may assume that the map $H$ 
is real-analytic in a neighborhood of the closure of  $V\times W^3$. 
We also introduce, for every $\mu \in \N^N$,  new independent complex variables  $\Lambda^\mu \in \C^{N'}$ and $\widehat \Lambda^\mu \in \C^{N'}$, and write $\Lambda =(\Lambda^\mu)_{|\mu|\leq 2sN'}$, $\widehat \Lambda =(\widehat \Lambda^\mu)_{|\mu|\leq 2sN'}$.  For any open set $\Omega$ in some real manifold, we  write $\6A (\overline{\Omega})$ for the ring of real-analytic functions in a neighborhood of $\overline{\Omega}$. For $1\leq i,j\leq J$, $1\leq \nu \leq N'$,  we set 
\begin{equation}\label{e:def}
\begin{aligned}
R_{i,j,\nu}(t,p,q, Z, \Lambda, \widehat \Lambda ):=&P_{i,\nu}  (H (t,p,q,Z), \overline{\Lambda}, \Lambda )\ D_{j}(H (t,p,q,Z), \overline{ \widehat \Lambda}, \widehat \Lambda)\\
&-P_{j,\nu}(H (t,p,q, Z), \overline{\widehat \Lambda},\widehat \Lambda) \ D_{i}(H (t,p,q, Z), \overline{\Lambda}, \Lambda),
\end{aligned}
\end{equation}
and also define, for each $\alpha \in \N^N$,  
$$
	R^\alpha_{i,j,\nu}(t,p,q, \Lambda, \widehat \Lambda)
	:=
	\frac{
		\partial^{|\alpha|} R_{i,j,\nu}
	}{
		\partial Z^\alpha
	}
	(t,p,q, q,\Lambda, \widehat \Lambda)
	\in 
	\6A \left (\overline{V \times W^2} \right) \left [\Lambda , \widehat \Lambda, \overline{\Lambda}, \overline{\widehat \Lambda}\right].
$$
For $1\leq i,j\leq J$ and $1\leq \nu \leq N'$, let $\6I_{i,j,\nu}$ be the ideal generated by the $R^\alpha_{i,j,\nu}$ for $\alpha \in \N^N$ in the ring $\6R:=\6A \left (\overline{V \times W^2} \right) \left [\Lambda ,\overline{\Lambda},\widehat \Lambda, \overline{\widehat \Lambda} \right]$. By \cite{Fri}, the ring  $\6A(\overline{V \times W^2})$ is noetherian, and therefore, so is $\6R$.  Hence there is an integer $\ell_{i,j,\nu}$ such that $\6I_{i,j,\nu}$ is generated, as an ideal in $\6R$, by the $R^\alpha_{i,j,\nu}$ for $|\alpha|\leq \ell_{i,j,\nu}$. Set $K={\rm max} \{\ell_{i,j,\nu}:1\leq i,j\leq J,\, 1\leq \nu \leq N'\}$. We claim that the conclusion of the proposition holds with $U_0=W$ and the above mentioned choice of $K$.

Indeed,  pick $q\in M\cap W$ and assume that $f,g\colon (\C^{N},q)\to \C^{N'}$ are two germs of holomorphic maps sending $M$ into $\4S^{2N'-1}$, with $j_q^{K}f=j^{K}_qg$. It follows from Proposition \ref{p:leader} that we may find $1\leq j_1,j_2\leq J$ such that
\begin{equation}
f(Z)= \frac{P_{j_1}\left (H (t,p,q,Z), (\partial^\mu \bar f(\bar p), \partial^\mu f(p))_{|\mu|\leq 2sN'} \right )}{D_{j_1}\left ((H (t,p,q, Z), (\partial^\mu \bar f(\bar p), \partial^\mu f(p))_{|\mu| \leq 2sN'}) \right )}
\end{equation}
\begin{equation}
g(Z)= \frac{
	P_{j_2}\left (H (t,p, q,Z), (\partial^\mu \bar g(\bar p), \partial^\mu g(p))_{|\mu|\leq 2sN'} \right )
}{
	D_{j_2}\left ((H (t,p, q, Z), (\partial^\mu \bar g(\bar p), \partial^\mu g(p))_{|\mu| \leq 2sN'}) \right )
}\\ 
\end{equation}
for $Z$ sufficiently close to $q$, $t\in V$, and $p\in M$ sufficiently close to $q$. Since $f(Z)-g(Z)=O(|Z-q|^{K+1})$, we get that for all $t$, $Z$ and $p$ as above, and for $1\leq \nu \leq N'$, 
$$R_{j_1,j_2,\nu}(t,p,q, Z,( \partial^\mu f(p))_{|\mu|\leq 2sN'},( \partial^\mu g(p))_{|\mu|\leq 2sN'})=O(|Z-q|^{K+1}),$$
or equivalently that 
$$R^\alpha _{j_1,j_2,\nu}(t,p,q,(\partial^\mu f(p))_{|\mu|\leq 2sN'},( \partial^\mu g(p))_{|\mu|\leq 2sN'})=0,\ |\alpha| \leq K.$$
By the choice of $K=\ell^0$, we get that for $\nu=1,\ldots,N'$,
$$R_{j_1,j_2,\nu}(t,p,q, Z, (\partial^\mu f(p))_{|\mu|\leq 2sN'},( \partial^\mu g(p))_{|\mu|\leq 2sN'})=0$$
for $Z\in W$, $t\in V$, and $p\in M$ sufficiently close to $q$, which implies that $f(Z)=g(Z)$ for $Z$ close to $q$, i.e. $f=g$.
\end{proof}

Since any real-analytic CR submanifold in $\C^N$ is locally biholomorphically equivalent to a product manifold $\Sigma\times \{0\}\subset \C^{N-e}\times \C^e$ for some real-analytic generic submanifold $\Sigma \subset \C^{N-e}$, the following result follows at once from Proposition \ref{p:dobetter} and \cite{MMZ03}.
\begin{Pro}\label{p:tired}
Let $M\subset \C^N$ be a real-analytic minimal CR submanifold through the origin. Then there exists a neighbhorhood $M_0$ of $0$ in $M$ and an integer $K>0$ such that for every $q\in M_0$, if 
$f,g\colon (M,q)\to \4S^{2N'-1}$ 
are two germs of $\6C^\infty$-smooth CR maps with $j_q^Kf=j_q^Kg$, then $f=g$.
\end{Pro}


Theorem \ref{t:mainbest} is  then a straightforward consequence of Proposition \ref{p:tired}.

\section*{Acknowledgements} The authors would like to thank the referees for their remarks which improved the readability of the paper.


\begin{thebibliography}{}

	

\bibitem{Alex}
   H. Alexander : Holomorphic mappings from the ball and polydisc, {\em Math. Ann.}, {\bf 209}, (1974), 249--256.

\bibitem{BERacta}
   M.S.~Baouendi; P.~Ebenfelt; L.P.~Rothschild : Algebraicity of holomorphic mappings between real algebraic in $\C^n$, {\em Acta Math.}, {\bf 177}, (1996), 225--273.

\bibitem{BERcag}
   M.S.~Baouendi; P.~Ebenfelt; L.P.~Rothschild : CR automorphisms of real analytic manifolds in complex  space, {\em Comm. Anal. Geom.}, {\bf 6}, (1998), 291--315.


\bibitem{BERbook}
   M.S.~Baouendi; P.~Ebenfelt; L.P.~Rothschild :
  {\em Real Submanifolds in Complex Space and
  Their Mappings}. Princeton Math. Series {\bf 47},
  Princeton Univ. Press, 1999.



\bibitem{BER99}
   M.S.~Baouendi; P.~Ebenfelt; L.P.~Rothschild : Rational dependence of smooth and analytic CR mappings on their jets, {\em Math. Ann.}, {\bf 315}, (1999), 205--249.
   
   

\bibitem{BER00}
   M.S.~Baouendi; P.~Ebenfelt; L.P.~Rothschild : Convergence and finite determination of formal CR mappings, {\em J. Amer. Math. Soc.}, {\bf 13}, (2000), 697--723.

\bibitem{BH05}
   M.S.~Baouendi, X. Huang : Super-rigidity for holomorphic mappings between hyperquadrics with positive signature, {\em J. Diff. Geom.}, {\bf 69}, (2005), 379--398.

\bibitem{BMR} M.S. Baouendi, N. Mir, L.P. Rothschild: Reflection ideals and mappings between generic submanifolds in complex space, {\em J. Geom. Anal.}, {\bf 12}(4), (2002), 543--580.

\bibitem{BRZ01} M.S. Baouendi, L.P. Rothschild, D. Zaitsev: Equivalences of real submanifolds in complex space {\em J. Diff. Geom}, {\bf 59}, (2001), 301--351.





\bibitem{BX15}
   S.~Berhanu; M.~Xiao : On the $C^\infty$ version of the reflection principle for mappings between CR manifolds, {\em Amer. J. Math.}, {\bf 137}(5), (2015), 1365--1400.


\bibitem{BB}
   F. Bertrand, L. Blanc-Centi : Stationary holomorphic discs and finite jet determination problems, {\em Math. Ann.}, {\bf 358}, (2014), 477--509.


\bibitem{BBM}
   F. Bertrand, L. Blanc-Centi, F. Meylan : Stationary discs and finite jet determination for non-degenerate generic real submanifolds, {\em Adv. Math.}, {\bf 343}, (2019), 910--934.



\bibitem{BDL}
   F. Bertrand, G. Della Sala, B. Lamel : Jet determination of smooth CR automorphisms and generalized stationary discs, {\em Math. Z.}, {\bf 294}, (2020), 1611--1634.




\bibitem{Cartan} E. Cartan : Sur la g\'eom\'etrie pseudo-conforme des hypersurfaces de deux variables complexes, I.  {\em Ann. Mat. Pura Appl.},  {\bf 11}, (1932), 17--90; Part II, {\em Ann. Sc. Norm. Sup. Pisa}, {\bf 1}, (1932), 333--354.


\bibitem{CM} S. S. Chern; J.K. Moser : Real hypersurfaces in complex manifolds  {\em Acta Math.},  {\bf 133}, (1974), 219--271.

\bibitem{CS} J. Cima, T.J. Suffridge : Boundary behavior of rational proper maps,  {\em Duke Math. J.},  {\bf 60}, (1990), 135--138.

\bibitem{Chia} S.~Chiappari : Holomorphic extension of proper meromorphic mappings  {\em Michigan Math. J.},  {\bf 38}, (1991), 167--174.

\bibitem{DF} K. Diederich, J.E. Fornaess : Pseudoconvex domains with real-analytic boundary. 
{\em Ann. Math.}, 107(3), (1978), 371--384. 

\bibitem{Ebasian} P. Ebenfelt : Finite jet determination of holomorphic mappings at the boundary,  {\em Asian J. Math.},  {\bf 5}, (2001), 637--662.


\bibitem{EHZ1} P. Ebenfelt, X. Huang, D. Zaitsev : Rigidity of CR-immersions into spheres,  {\em Comm. Anal. Geom.},  {\bf 12}, (2004), 631--670.


\bibitem{EHZ2} P. Ebenfelt, X. Huang, D. Zaitsev : The equivalence problem and rigidity for hypersurfaces embedded into hyperquadrics,  {\em Amer. J. Math.},  {\bf 127}, (2005), 169--191.

\bibitem{EL} P. Ebenfelt, B. Lamel : Finite jet determination of CR embeddings,  {\em J. Geom.  Anal.},  {\bf 14}, (2004), 241--265.

\bibitem{ELZ} P. Ebenfelt, B. Lamel, D. Zaitsev : Finite jet determination of local analytic CR automorphisms and their parametrization by 2-jets in the finite type case,  {\em Geom. Funct. Anal.},  {\bf 13}, (2003), 546--573.

\bibitem{Faran} J. Faran : The linearity of proper holomorphic maps between balls in the low codimension case, {\em J. Diff. Geom.}, {\bf 24}, (1986), 15--17.

\bibitem{F89} F. Forstneri\v c : Extending proper holomorphic mappings of positive codimension, {\em Invent. math.}, {\bf 95}, (1989), 31--62.


\bibitem{Fri} J. Frisch :  Points de platitude d'un morphisme d'espaces analytiques complexes, {\em Invent. math.}, {\bf 4}, (1967), 118--138.



\bibitem{H99} X. Huang :  On a linearity problem for proper holomorphic maps between balls in complex spaces of different dimensions, {\em J. Diff. Geom}, {\bf 51}, (1999), 13--33.

\bibitem{H03} X. Huang :  On a semi-rigidity property for holomorphic maps, {\em Asian J. Math.}, {\bf 7}, (2003), 463--492.



\bibitem{HJ98} X. Huang, S. Ji :  Global holomorphic extension of a local map and a Riemann mapping theorem for algebraic domains, {\em Math. Res. Lett.}, {\bf 5}, (1998), 247--260.

\bibitem{HYsurvey} X. Huang, W. Yin :  On some rigidity problems in Cauchy-Riemann analysis, {\em Proceedings of the International Conference on Complex Geometry and Related Fields,}, 89--107, AMS/IP Stud. Adv. Math., 39, Amer. Math. Soc., Providence, RI, 2007.

\bibitem{Juhlin} R. Juhlin : Determination of formal CR mappings by a finite jet, {\em Adv. Math.} {\bf 222} (2009), 1611--1648.

\bibitem{JL13} R. Juhlin, B. Lamel : Automorphism groups of minimal real-analytic CR manifolds, {\em J. Eur. Math. Soc.} {\bf 15} (2013), 509--537.

\bibitem{KZ} S.-Y. Kim, D. Zaitsev : Equivalence and embedding problems for CR-structures of any codimension, {\em Topology} {\bf 44}, (2005), 557--584.

\bibitem{KZinvent} S.-Y. Kim, D. Zaitsev :  Rigidity of CR maps between Shilov boundaries of bounded symmetric domains, {\em Invent. math.} {\bf 193}, (2013), 409--437.

\bibitem{KZma} S.-Y. Kim, D. Zaitsev : Rigidity of proper holomorphic maps between bounded symmetric domains, {\em Math. Ann.} {\bf 362}, (2015), 639--677.

\bibitem{KMZ} M. Kolar, F. Meylan, D. Zaitsev : Chern-Moser operators and polynomial models in CR geometry, {\em Adv. Math.} {\bf 263}, (2014), 321--356.

\bibitem{La01} B. Lamel : Holomorphic maps of real submanifolds in complex spaces of different dimensions, {\em Pacific J. Math.} {\bf 201}(2), (2001), 357--387.

\bibitem{LM07} B. Lamel, N. Mir : Parametrization of local CR automorphisms by finite jets and applications,  {\em J. Amer. Math. Soc.},  {\bf 20}, (2007), 519--572.

\bibitem{LM17} B. Lamel, N. Mir : Convergence of formal CR mappings into strongly pseudoconvex Cauchy-Riemann manifolds,  {\em Invent. math.},  {\bf 210}, (2017), 963--985.


\bibitem{MMZ03} F. Meylan, N. Mir, D. Zaitsev : Holomorphic extension of smooth CR mappings between real-analytic and real-algebraic CR-manifolds, {\em Asian J. Math.} {\bf 7}(4), (2003), 493--509.


\bibitem{Mi03} N. Mir : Analytic regularity of CR maps into spheres, {\em Math. Res. Lett.},  {\bf 10}, (2003), 447--457.


\bibitem{Po} H. Poincar\'e, Les fonctions analytiques de deux variables et la repr\'esentation conforme, {\em Rend. Circ. Mat. Palermo, II,} {\bf 23}, (1907), 185--220.



\bibitem{Pi} S. Pinchuk  :  Analytic continuation of
holomorphic mappings and problems of holomorphic classification of
multi-dimensional domains, (Russian) {\em Math. Zam.} {\bf 33}
(1983), 301--314. English Translation in {\em Math. Notes} {\bf
33} (1983), 151--157.



\bibitem{Pin} S. Pinchuk : Analytic continuation of mappings along strictly pseudo-convex hypersurfaces, (Russian)  {\em Dokl. Akad. Nauk SSSR} {\bf 236} (1977), 544--547. English translation in {\em Soviet Math Dokl.}, {\bf 18}, (1978), 1237--1240.


\bibitem{Tan} N. Tanaka : On the pseudo-conformal geometry of hypersurfaces of the space of $n$ complex variables. {\em J. Math. Soc. Japan} {\bf 14}, (1962), 397--429.

\bibitem{T90} A.E. Tumanov : Extension of CR-functions into a wedge. {\em Mat. Sb.} {\bf 181} (1990), no. 7, 951--964; translation in {\em Math. USSR-Sb.} {\bf 70}, no. 2, (1991), 385--398.

\bibitem{VEK} A.G. Vitushkin, V.V. Ezhov, N.G. Kruzhilin : Continuation of holomorphic mappings along real-analytic hypersurfaces. {\em Current problems in mathematics. Mathematical analysis, algebra, topology, Trudy Mat. Inst. Steklov} {\bf 167}, (1985), 60--95.

\bibitem{W77} S.M. Webster : On the mapping problem for algebraic real hypersurfaces, {\em Invent. math.} {\bf 43}, (1977), 53--68.


\bibitem{W79} S.M. Webster : On  mapping an $n$-ball into an $(n+1)$-ball in complex spaces,  {\em Pacific J. Math.} {\bf 81}, (1979), 267--272.

\bibitem{Z97} D. Zaitsev : Germs of local automorphisms of real-analytic CR structures and analytic dependence on $k$-jets. {\em Math. Res. Lett.} {\bf 4}, (1997), 823--842.



\bibitem{Z99} D. Zaitsev : Algebraicity of local holomorphisms between real-algebraic submanifolds of complex spaces. {\em Acta Math.} {\bf 183}, (1999), 273--305.


\end{thebibliography}
\end{document}